\documentclass[reqno,10pt, centertags]{amsart}
\usepackage{amssymb,upref,esint,float,color}
\usepackage{centernot}

\usepackage{hyperref}
\newcommand*{\mailto}[1]{\href{mailto:#1}{\nolinkurl{#1}}}
\newcommand{\arxiv}[1]{\href{http://arxiv.org/abs/#1}{arXiv:#1}}

\makeatletter
\def\theequation{\@arabic\c@equation}

\newcommand{\bbN}{{\mathbb{N}}}
\newcommand{\bbR}{{\mathbb{R}}}

\newcommand{\bbC}{{\mathbb{C}}}

\newcommand{\cB}{{\mathcal B}}
\newcommand{\cC}{{\mathcal C}}

\newcommand{\cH}{{\mathcal H}}
\newcommand{\cI}{{\mathcal I}}

\newcommand{\no}{\nonumber}
\newcommand{\lb}{\label}
\newcommand{\f}{\frac}

\newcommand{\Oh}{O}
\newcommand{\oh}{o}

\newcommand{\tr}{\text{\rm{tr}}}

\newcommand{\ran}{\text{\rm{ran}}}

\newcommand{\dom}{\text{\rm{dom}}}

\newcommand{\slim}{\text{\rm{s-lim}}}

\newcommand{\bi}{\bibitem}

\renewcommand{\Im}{\text{\rm Im}}
\renewcommand{\ln}{\text{\rm ln}}

\renewcommand{\dot}{\overset{\textbf{\Large.}}}

\numberwithin{equation}{section}

\newtheorem{theorem}{Theorem}[section]
\newtheorem{lemma}[theorem]{Lemma}

\newtheorem{hypothesis}[theorem]{Hypothesis}
\newtheorem{example}[theorem]{Example}
\theoremstyle{remark}
\newtheorem{remark}[theorem]{Remark}

\begin{document}

\title[Traces and Modified Fredholm Determinants]{On Traces and Modified Fredholm
Determinants for Half-Line Schr\"odinger Operators \\ with Purely Discrete Spectra}

\author[F.\ Gesztesy]{Fritz Gesztesy}
\address{Department of Mathematics,
Baylor University, One Bear Place \#97328,
Waco, TX 76798-7328, USA}
\email{\mailto{Fritz\_Gesztesy@baylor.edu}}
\urladdr{\url{http://www.baylor.edu/math/index.php?id=935340}}

\author[K.\ Kirsten]{Klaus Kirsten}
\address{GCAP-CASPER, Department of Mathematics,
Baylor University, One Bear Place \#97328,
Waco, TX 76798-7328, USA}
\email{\mailto{Klaus\_Kirsten@baylor.edu}}
\urladdr{\url{http://www.baylor.edu/math/index.php?id=54012}}

\date{\today}
\thanks{K.K.\ was supported by the Baylor University Summer Sabbatical Program.}
\thanks{To appear in {\it The Quarterly of Applied Mathematics.}}

\begin{abstract}
After recalling a fundamental identity relating traces and modified Fredholm determinants, we apply it to a class of half-line Schr\"odinger operators $(- d^2/dx^2) + q$ on $(0,\infty)$ with purely discrete spectra.  Roughly speaking, the class considered is generated by potentials $q$ that, for some fixed $C_0 > 0$,
$\varepsilon > 0$, $x_0 \in (0, \infty)$, diverge at infinity in the manner that $q(x) \geq C_0 x^{(2/3) + \varepsilon_0}$ for all $x \geq x_0$. We treat all self-adjoint boundary conditions at the left endpoint $0$.
\end{abstract}

\maketitle



\section{Introduction} \lb{s1}

To set the stage for describing the principal purpose of this note, we assume that $q$ satisfies
$q \in L^1_{loc} (\bbR_+; dx)$, $q$ real-valued a.e.\ on $\bbR_+$, and that for some $ \varepsilon_0 > 0$,
$C_0 >0$,  and sufficiently large $x_0 > 0$,
\begin{equation}
q(x) \geq C_0 \, x^{(2/3) + \varepsilon_0}, \quad x \in (x_0,\infty).      \lb{1.1}
\end{equation}
Next, we introduce the half-line Schr\"odinger operator $H_{+,\alpha}$ in $L^2(\bbR_+; dx)$ as the
$L^2$-realization of the differential expression $\tau_+$ of the type
\begin{equation}
\tau_+ = - \f{d^2}{dx^2} + q(x) \, \text{ for a.e.~$x \in \bbR_+$}     \lb{1.2}
\end{equation}
(here $\bbR_+ = (0,\infty)$), and a self-adjoint boundary condition of the form
\begin{equation}
\sin(\alpha) g'(0) + \cos(\alpha) g(0) = 0, \quad \alpha \in [0,\pi)    \lb{1.3}
\end{equation}
for $g$ in the domain of $H_{+,\alpha}$. 
Then under appropriate additional technical assumptions on $q$ (cf.\ Hypothesis \ref{h3.1}), we will prove in Theorem \ref{t3.3} that
\begin{align}
& \tr_{L^2(\bbR_+;dx)} \big((H_{+,\alpha} - z I_+)^{-1} - (H_{+,\alpha} - z_0 I_+)^{-1}\big)    \no \\
& \quad = - \f{d}{dz} \ln\big(\det{_{2,L^2(\bbR_+;dx)}} \big(I_+ - (z - z_0) (H_{+,\alpha } - z_0 I_+)^{-1}\big)\big)
\no \\
& \quad = \f{d}{dz} \ln\big(\sin (\alpha) f_{+,1} ' (z,0,x_0)+\cos (\alpha )  f_{+,1}(z,0,x_0)\big)\bigg|_{z=z_0} \no\\
& \qquad - \f{d}{dz} \ln\big(\sin (\alpha ) f_{+,1}' (z,0,x_0) + \cos (\alpha ) f_{+,1}(z,0,x_0)\big)    \no \\
& \qquad +\f{1}{2}  \cI (z,z_0,x_0),    \lb{1.4}
\end{align}
(with $I_+$ abbreviating the identity operator in $L^2(\bbR_+; dx)$) and
\begin{align}
& \det{_{2,L^2(\bbR_+;dx)}} \big(I_+ - (z - z_0) (H_{+,\alpha } - z_0 I_+)^{-1}\big)    \no \\
& \quad = \bigg[\f{ \sin (\alpha ) f_{+,1} ' (z,0,x_0) + \cos (\alpha ) f_{+,1} (z,0,x_0) }{ \sin (\alpha ) f_{+,1} ' (z_0,0,x_0) + \cos (\alpha ) f_{+,1} (z_0,0,x_0)}\bigg]   \no \\
& \qquad \times \exp{\bigg( - (z - z_0) \,\,\f{ \sin (\alpha ) \dot{f^{\, \prime}}_{\!\!\!\!+,1} (z_0,0,x_0)+\cos (\alpha ) \dot{f}_{+,1} (z_0,0,x_0) }
{ \sin (\alpha ) f_{+,1}'(z_0,0,x_0)+\cos (\alpha ) f_{+,1} (z_0, 0,x_0) } \bigg)}    \lb{1.5}  \\
& \qquad \times \exp\bigg(- \f{1}{2} \int_{z_0}^z d \zeta \, \cI(\zeta,z_0,x_0)\bigg).    \no
\end{align}
Here we abbreviated $\prime = d/dx$, $\dot{} = d/dz$, 
\begin{equation}
 \cI(z,z_0,x_0) = \int_{x_0}^{\infty} dx \big\{[q(x) - z]^{-1/2} - [q(x) - z_0]^{-1/2}\big\},   \lb{1.6}
\end{equation}
and $f_{+,1}(z,x,x_0)$ represents an analog of the Jost solution in the case where $q$ denotes a short-range potential (i.e., one that decays sufficiently fast as $x \to \infty$). Finally, $\det_2(\, \cdot \,)$ abbreviates the modified Fredholm determinant naturally associated with Hilbert--Schmidt operators. 

Following the recent paper by Menon \cite{Me16}, which motivated us to write the present note, we then revisit the exactly solvable  example $q(x) = x$, $x \in \bbR_+$, in Example \ref{e3.4}, 

In our final result, Theorem \ref{t3.5}, we will also treat the case of different boundary condition
parameters $\alpha_j \in [0,\pi)$, $j=1,2$, and derive the following extension of \eqref{1.4},
\begin{align}
& {\tr}_{L^2(\bbR_+; dx)} \big((H_{+,\alpha_2} - z I_+)^{-1}
- (H_{+,\alpha_1} - z_0 I_+)^{-1}\big)    \no \\
& \quad = - \f{d}{dz} \ln\bigg(\f{\sin(\alpha_2) f_{+,1}'(z,0,x_0)
+ \cos(\alpha_2) f_{+,1}(z ,0,x_0)}{\sin(\alpha_1) f_{+,1}'(z_0 ,0,x_0)
+ \cos(\alpha_1) f_{+,1}(z_0 ,0,x_0)}\bigg),    \lb{1.7} \\
& \qquad +\f{1}{2}  \cI(z,z_0,x_0).       \no
\end{align}

Our proofs of \eqref{1.4}, \eqref{1.5}, and \eqref{1.6} in Section \ref{s3} are based on fundamental  connections between traces and modified Fredholm determinants briefly discussed in Section \ref{s2}, in particular, we will employ the relation (with $I_{\cH}$ the identity operator in $\cH$) 
\begin{align}
\begin{split}
& \tr_{\cH}\big((A - z I_{\cH})^{-1} - (A - z_0 I_{\cH})^{-1}\big)    \\
& \quad = - (d/dz) \ln\big({\det}_{\cH, 2}\big(I_{\cH} - (z - z_0) (A - z_0 I_{\cH})^{-1}\big)\big),    \lb{1.8}
\end{split}
\end{align}
where $A$ denotes a densely defined and closed operator in $\cH$ with $\rho(A) \neq \emptyset$, and
$(A - z I_{\cH})^{-1} \in \cB_2(\cH)$, $z \in \rho(A)$.

Finally, we briefly summarize some of the basic notation used in this paper.
Let $\cH$ be a separable, complex Hilbert space, $(\,\cdot\,,\,\cdot\,)_{\cH}$ the scalar product
in $\cH$ (linear in the second factor), and
$I_{\cH}$ the identity operator in $\cH$. The domain and range of an operator $T$ are
denoted by $\dom(T)$ and $\ran(T)$, respectively.
The kernel (null space) of $T$
is denoted by $\ker(T)$. The spectrum, point spectrum, and resolvent set of a closed linear
operator in $\cH$ will be denoted by $\sigma(\cdot)$, $\sigma_p(\cdot)$, and $\rho(\cdot)$; the
discrete spectrum of $T$ (i.e., points in $\sigma_p(T)$ which are isolated from the rest of
$\sigma(T)$, and which are eigenvalues of $T$ of finite algebraic multiplicity) is
abbreviated by $\sigma_d(T)$. The {\it algebraic multiplicity} $m_a(z_0; T)$ of an eigenvalue
$z_0\in\sigma_d(T)$ is the dimension of the range of the corresponding {\it Riesz projection}
$P(z_0;T)$,
\begin{equation}
m_a(z_0; T) = \dim(\ran(P(z_0;T))) = \tr_{\cH}(P(z_0;T)),
\end{equation}
where (with the symbol $ \ointctrclockwise$ denoting counterclockwise oriented
contour integrals)
\begin{equation}
 P(z_0;T)=\f{-1}{2\pi i} \ointctrclockwise_{C(z_0;\varepsilon)} d\zeta \,
(T - \zeta I_{\cH})^{-1},
\end{equation}
for $0 < \varepsilon<\varepsilon_0$ and  $D(z_0;\varepsilon_0) \backslash \{z_0\}\subset \rho(T)$; here
$D(z_0; r_0) \subset \bbC$ is the open disk with center $z_0$ and radius
$r_0 > 0$, and $C(z_0; r_0) = \partial D(z_0; r_0)$ the corresponding circle.

The Banach spaces of bounded
and compact linear operators in $\cH$ are denoted by $\cB(\cH)$ and
$\cB_\infty(\cH)$, respectively. Similarly, the Schatten--von Neumann
(trace) ideals will subsequently be denoted by $\cB_p(\cH)$,
$p \in [1,\infty)$. In addition,
$\tr_{\cH}(T)$ denotes the trace of a trace class operator $T\in\cB_1(\cH)$,
${\det}_{\cH} (I_{\cH} - T)$ the Fredholm determinant of $I_{\cH} - T$, and for $p \in \bbN$,
$p \geq 2$, ${\det}_{\cH,p} (I_{\cH} - T)$ abbreviates the $p$th modified Fredholm determinant of $I_{\cH} - T$.

\section{Traces and (Modified) Fredholm Determinants of Operators} \lb{s2}

In this section we recall some well-known formulas relating traces and (modified) Fredholm
determinants. For background
on the material used in this section see, for instance, \cite{GGK96}, \cite{GGK97}, \cite[Ch.\ XIII]{GGK00},
\cite[Ch.~IV]{GK69},  \cite[Ch.\ 17]{RS78}, \cite{Si77}, \cite[Ch.\ 3]{Si05}.

To set the stage we start with densely defined, closed, linear operators $A$ in $\cH$ having a trace class resolvent, and hence introduce the following assumption:

\begin{hypothesis} \lb{h2.1}
Suppose that $A$ is densely defined and closed in $\cH$ with $\rho(A) \neq \emptyset$, and
$(A - z I_{\cH})^{-1} \in \cB_1(\cH)$ for some $($and hence for all\,\footnote{One applies the resolvent equation for $A$ and the binomial theorem.}$)$ $z \in \rho(A)$.
\end{hypothesis}

Given Hypothesis \ref{h2.1} and $z_0 \in \rho(A)$, consider the bounded, entire family $A(\, \cdot \,)$ defined by
\begin{equation}
A(z) := I_{\cH} - (A - z I_{\cH})(A - z_0 I_{\cH})^{-1} =  (z - z_0) (A - z_0 I_{\cH})^{-1}, \quad
z \in \bbC.     \lb{2.1}
\end{equation}
Employing the formula (cf.\ \cite[Sect.~IV.1]{GK69}, see also \cite[Sect.~I.7]{Ya92}),
\begin{equation}
\tr_{\cH}\big((I_{\cH} - T(z))^{-1} T'(z)\big) = - (d/dz) \ln({\det}_{\cH}(I_{\cH} - T(z))),   \lb{2.2}
\end{equation}
valid for a trace class-valued analytic family $T(\, \cdot \,)$ on an open set $\Omega \subset \bbC$
such that $(I_{\cH} - T(\, \cdot \,))^{-1} \in \cB(\cH)$, and applying it to the entire family $A(\, \cdot \,)$ then results in
\begin{align}
\tr_{\cH}\big((A - z I_{\cH})^{-1}\big) &= - (d/dz) \ln\big({\det}_{\cH}\big(I_{\cH} - (z - z_0) (A - z_0 I_{\cH})^{-1}\big)\big)   \no \\
&= - (d/dz) \ln\big({\det}_{\cH}\big((A - z I_{\cH}) (A - z_0 I_{\cH})^{-1}\big)\big),   \lb{2.3} \\
& \hspace*{5.75cm} z \in \rho(A).    \no
\end{align}

One notes that the left- and hence the right-hand side of \eqref{2.3} is independent of the choice
of $z_0 \in \rho(A)$.

Next, following the proof of \cite[Theorem~3.5\,(c)]{Si05} step by step, and employing a
Weierstrass-type product formula (see, e.g., \cite[Theorem~3.7]{Si05}), yields the subsequent
result (see also \cite{GW95}).

\begin{lemma} \lb{l2.2}
Assume Hypothesis \ref{h2.1} and let $\lambda_k \in \sigma(A)$ then
\begin{equation} \lb{2.4}
{\det}_{\cH} \big(I_{\cH} - (z-z_0)(A - z_0 I_{\cH})^{-1}\big) = (\lambda_k-z)^{m_a(\lambda_k)}
[C_k +O(\lambda_k - z)], \quad C_k\neq0
\end{equation}
as $z$ tends to $\lambda_k$, that is, the multiplicity of the zero
 of the Fredholm determinant ${\det}_{\cH}\big(I_{\cH} - (z-z_0)(A - z_0 I_{\cH})^{-1}\big)$
at $z=\lambda_k$ equals the algebraic multiplicity of the eigenvalue $\lambda_k$ of $A$.

In addition, denote the spectrum of $A$
by $\sigma(A)=\{\lambda_k\}_{k\in\bbN}$, $\lambda_k \neq \lambda_{k'}$
for $k \neq k'$. Then
\begin{align}
\begin{split}
{\det}_{\cH}(I_{\cH} - (z-z_0)(A - z_0 I_{\cH})^{-1}) &= \prod_{k \in \bbN}
\big[1-(z-z_0)(\lambda_k - z_0)^{-1}\big]^{m_a(\lambda_k)} \\
&= \prod_{k \in \bbN} \bigg(\f{\lambda_k - z}{\lambda_k-z_0}
\bigg)^{m_a(\lambda_k)},    \lb{2.5}
\end{split}
\end{align}
with absolutely convergent products in \eqref{2.5}.
\end{lemma}

The case of trace class resolvent operators is tailor-made for a number of one-dimensional
Sturm--Liouville operators (e.g., finite interval problems). But for applications to half-line problems with potentials behaving like $x$, or increasing slower than $x$ at $+\infty$, and similarly for partial differential operators, traces of higher-order powers of resolvents need to be involved which naturally lead to modified Fredholm determinants as follows.

\begin{hypothesis} \lb{h2.3}
Let $p \in \bbN$, $p \geq 2$, and suppose that $A$ is densely defined and closed in $\cH$ with $\rho(A) \neq \emptyset$, and
$(A - z I_{\cH})^{-1} \in \cB_p(\cH)$ for some $($and hence for all\,$)$ $z \in \rho(A)$.
\end{hypothesis}

Applying the formula
\begin{equation}
\tr_{\cH}\big((I_{\cH} - T(z))^{-1} T(z)^{p-1} T'(z)\big) = - (d/dz) \ln({\det}_{\cH, p}(I_{\cH} - T(z))),   \lb{2.6}
\end{equation}
valid for a $\cB_p(\cH)$-valued analytic family $T(\, \cdot \,)$ on an open set $\Omega \subset \bbC$
such that $(I_{\cH} - T(\, \cdot \,))^{-1} \in \cB(\cH)$, \cite[Sect.~IV.2]{GK69} (see also
\cite[Sect.~I.7]{Ya92}) to the entire family $A(\, \cdot \,)$ in \eqref{2.1}, assuming Hypothesis \ref{h2.3},  then yields
\begin{align}
& (z - z_0)^{p-1} \tr_{\cH}\big((A - z I_{\cH})^{-1}(A - z_0 I_{\cH})^{1-p}\big)    \no \\
& \quad = - (d/dz) \ln\big({\det}_{\cH, p}\big(I_{\cH} - (z - z_0) (A - z_0 I_{\cH})^{-1}\big)\big),    \lb{2.7} \\
& \quad = - (d/dz) \ln\big({\det}_{\cH, p}\big((A - z I_{\cH}) (A - z_0 I_{\cH})^{-1}\big)\big),  \quad
z \in \rho(A).    \no
\end{align}

In the special case $p=2$ this yields
\begin{align}
\begin{split}
& \tr_{\cH}\big((A - z I_{\cH})^{-1} - (A - z_0 I_{\cH})^{-1}\big)    \\
& \quad = - (d/dz) \ln\big({\det}_{\cH, 2}\big(I_{\cH} - (z - z_0) (A - z_0 I_{\cH})^{-1}\big)\big).    \lb{2.8}
\end{split}
\end{align}

We refer to Section \ref{s3} for an application of \eqref{2.8} to half-line Schr\"odinger operators with potentials diverging at infinity. For additional background and applications of (modified) Fredholm determinants to ordinary differential operators we also refer to \cite{BFK95}, \cite{CGNZ14},
\cite{DD78}, \cite{GK18}, \cite{GM03}, \cite{GZ12}, \cite{Ki01}--\cite{LS77}, \cite{OY12},
and the extensive literature cited therein.

\section{Schr\"odinger Operators on a Half-Line} \lb{s3}

We now illustrate \eqref{2.8} with the help of self-adjoint Schr\"odinger operators $-\frac{d^2}{dx^2} + q$ on the half-line $\bbR_+ = (0,\infty)$ in the particular case where the potential $q$ diverges at $\infty$ and hence gives rise to a purely discrete spectrum (i.e, the absence of essential spectrum).

To this end we introduce the following set of assumptions on $q$:

\begin{hypothesis} \lb{h3.1}
Suppose $q$ satisfies
\begin{equation}
q\in L^1_{loc} (\bbR_+; dx), \, \text{ $q$ is real-valued a.e.\ on $\bbR_+$,}    \lb{3.47}
\end{equation}
and for some $ \varepsilon_0 > 0$, $C_0 >0$,  and sufficiently large $x_0 > 0$,
\begin{align}
& q, q' \in AC([x_0, R]) \, \text{ for all $R > x_0$,}     \lb{3.51} \\
& q(x) \geq C_0 \, x^{(2/3) + \varepsilon_0}, \quad x \in (x_0,\infty),     \lb{3.52} \\
& q'/q \underset{x \to \infty}{=} \oh\big(q^{1/2}\big),       \lb{3.53}\\
& \big(q^{-3/2} q'\big)' \in L^1((x_0,\infty); dx).     \lb{3.54}
\end{align}
\end{hypothesis}

Condition \eqref{3.52} guarantees that traces and modified determinants in this paper (see, e.g., the ones in Theorem \ref{t3.3}) are well-defined. Conditions 
\eqref{3.53} and \eqref{3.54} are imposed so that \cite[Corollary~2.2.1]{Ea89} is applicable, implying the asymptotic behavior \eqref{3.55}.

Given Hypothesis \ref{h3.1}, we take $\tau_+$ to be the Schr\"odinger  differential expression
\begin{equation} \lb{3.2}
\tau_+= -\frac{d^2}{dx^2} + q(x) \, \text{ for a.e.~$x\in \bbR_+$,}
\end{equation}
and note that $\tau_+$ is regular at $0$ and in the limit point case at $+\infty$.
The \emph{maximal operator}
$H_{+,max}$ in $L^2(\bbR_+;dx)$ associated with $\tau_+$ is defined by
\begin{align}
&H_{+,max} f=\tau_+ f,    \no \\
& \, f \in \dom(H_{+,max})= \big\{g\in L^2(\bbR_+;dx) \, \big| \, g,  g' \in AC([0,b]) \, \text{for all $b > 0$};     \lb{3.3} \\
& \hspace*{3.5cm} \tau_+ g\in  L^2(\bbR_+;dx)\big\},   \no
\intertext{while the \emph{minimal operator} $H_{+,min}$ in
$L^2(\bbR_+;dx)$ associated with
$\tau_+$ is given by}
&H_{+,min} f=\tau_+ f,    \no \\
& \, f \in \dom(H_{+,min})= \big\{g\in L^2(\bbR_+;dx) \, \big| \, g,  g' \in
AC([0,b]) \, \text{for all $b > 0$};   \lb{3.4} \\
&\hspace*{3.45cm}g(0)=g'(0)=0; \, \tau_+ g\in L^2(\bbR_+;dx)\big\}.      \no
\end{align}

One notes that the operator $H_{+,min}$ is symmetric and that
\begin{equation}
H_{+,min}^*=H_{+,max}, \quad H_{+,max}^*=H_{+,min}      \lb{3.5}
\end{equation}
(cf., eg., \cite[Theorem 13.8]{We03}). 
Moreover, all self-adjoint extensions of $H_{+,min}$ are given by the
one-parameter family in $L^2(\bbR_+;dx)$
\begin{align}
&H_{+,\alpha} f=\tau_+ f,    \no \\
& \, f \in \dom(H_{+,\alpha})= \big\{g\in L^2(\bbR_+;dx) \, \big| \, g,  g' \in
AC([0,b]) \, \text{for all $b > 0$};   \lb{3.6} \\
&\hspace*{3.05cm}\sin(\alpha) g'(0) + \cos(\alpha) g(0) = 0; \, \tau_+ g\in L^2(\bbR_+;dx)\big\},      \no \\
& \hspace*{9.15cm} \alpha \in [0, \pi).  \no
\end{align}

Next, we introduce the fundamental system of solutions $\phi_{\alpha}(z, \, \cdot \,)$ and
$\theta_{\alpha}(z, \, \cdot \,)$, $\alpha \in [0,\pi)$, $z \in \bbC$, associated with $H_{+,\alpha}$
satisfying
\begin{equation}
(\tau_+ \psi(z,\, \cdot \,))(x) = z \psi(z,x), \quad z \in \bbC, \; x \in \bbR_+,    \lb{3.22}
\end{equation}
and the initial conditions
\begin{align}
\begin{split}
& \phi_{\alpha}(z,0) = - \sin(\alpha), \quad \phi_{\alpha}'(z,0) = \cos(\alpha),    \\
& \theta_{\alpha}(z,0) = \cos(\alpha), \qquad \, \theta_{\alpha}'(z,0) = \sin(\alpha).
\end{split}
\end{align}

Explicitly, one infers
\begin{align}
\begin{split}
\phi_{\alpha}(z,x)=\phi_{\alpha}^{(0)}(z,x)
+ \int_0^x dx' \, \frac{\sin(z^{1/2}(x-x'))}{z^{1/2}} q(x')
\phi_{\alpha}(z,x'),&    \lb{3.23} \\
z \in \bbC, \; \Im(z^{1/2}) \geq 0, \; x \geq 0,&
\end{split}
\end{align}
with
\begin{equation}
\phi_{\alpha}^{(0)}(z,x) = \cos(\alpha)\frac{\sin(z^{1/2}x)}{z^{1/2}}
- \sin(\alpha)\cos(z^{1/2}x), \quad
z \in \bbC, \; \Im(z^{1/2}) \geq 0, \; x \geq 0,
\end{equation}
and
\begin{align}
\begin{split}
\theta_{\alpha}(z,x)=\theta_{\alpha}^{(0)}(z,x)
+ \int_0^x dx' \, \frac{\sin(z^{1/2}(x-x'))}{z^{1/2}} q(x')
\theta_{\alpha}(z,x'),&    \lb{3.31} \\
z \in \bbC, \; \Im(z^{1/2}) \geq 0, \; x \geq 0,&
\end{split}
\end{align}
with
\begin{equation}
\theta_{\alpha}^{(0)}(z,x) = \cos(\alpha) \cos(z^{1/2} x) + \sin(\alpha)
\f{\sin(z^{1/2} x)}{z^{1/2}},
\quad z \in \bbC, \; \Im(z^{1/2}) \geq 0, \; x \geq 0.
\end{equation}

The Weyl--Titchmarsh solution, $\psi_{+,\alpha}(z, \, \cdot \,)$, and
Weyl--Titchmarsh $m$-function, $m_{+,\alpha}(z)$, corresponding to
$H_{+,\alpha}$, $\alpha \in [0, \pi)$, are then related via,
\begin{equation}
\psi_{+,\alpha}(z, \, \cdot \,) = \theta_{\alpha} (z, \, \cdot \,) + m_{+,\alpha}(z) \phi_{\alpha} (z, \, \cdot \,),
\quad z \in \rho(H_{+,\alpha}), \; \alpha \in [0,\pi),   \lb{3.30}
\end{equation}
where
\begin{equation}
\psi_{+,\alpha}(z, \, \cdot \,) \in L^2(\bbR_+; dx), \quad z \in \rho(H_{+,\alpha}), \;
\alpha \in [0,\pi).
\end{equation}

Let $I_+$ be the identity operator on $L^2(\bbR_+; dx)$. One then obtains for the Green's function $G_{+,\alpha}$ of $H_{+,\alpha}$
expressed in terms of $\phi_{\alpha}$ and $\psi_{+,\alpha}$,
\begin{align}
\begin{split}
& G_{+,\alpha}(z,x,x')
= (H_{+,\alpha} - z I_+)^{-1}(x,x')    \lb{3.57} \\
& \quad = \begin{cases}
\phi_{\alpha}(z,x) \, \psi_{+,\alpha} (z,x'), & 0\leq x\leq x' < \infty, \\
\phi_{\alpha}(z,x') \, \psi_{+,\alpha} (z,x), & 0\leq x' \leq x < \infty, \end{cases}
\quad z \in \rho(H_{+,\alpha}), \; \alpha \in [0,\pi),
\end{split}
\end{align}
utilizing
\begin{equation}
W(\theta_{\alpha} (z,\cdot), \phi_{\alpha}(z,\cdot)) =1, \quad z \in \bbC, \; \alpha \in [0,\pi),
\end{equation}
implying $W(\psi_{+,\alpha} (z,\cdot), \phi_{\alpha}(z,\cdot)) =1$, $z \in \rho(H_{+,\alpha})$. Here $W(f,g) = fg' - f' g$ abbreviates the Wronskian of suitable $f$ and $g$.

By \cite[Corollary~2.2.1]{Ea89}, Hypothesis \ref{h3.1} implies
the existence of two solutions $f_{+,j}(\lambda, \, \cdot \,,x_0)$, $j=1,2$, of
$\tau_+ \psi(\lambda, \, \cdot \,) = \lambda \psi(\lambda, \, \cdot \,)$,
$\lambda < 0$ sufficiently negative (and below $\inf(\sigma(H_{+,\alpha}))$), satisfying
\begin{align}
\begin{split}
& f_{+,j} (\lambda,x,x_0) \underset{x \to \infty}{=} 2^{-1/2} [q(x) - \lambda]^{-1/4}
\exp\bigg((-1)^{j}\int_{x_0}^x dx' [q(x') - \lambda]^{1/2}\bigg)    \\
& \hspace*{2.85cm} \times [1 + \oh(1)],     \\
& {f'_{+,j}} (\lambda,x,x_0) \underset{x \to \infty}{=} (-1)^j 2^{-1/2} [q(x) - \lambda]^{1/4}
\exp\bigg((-1)^{j}\int_{x_0}^x dx' [q(x') - \lambda]^{1/2}\bigg)    \\
& \hspace*{2.85cm} \times [1 + \oh(1)],   \quad  j=1,2,    \lb{3.55}
\end{split}
\end{align}
with
\begin{equation}
W\big(f_{+,1}(\lambda, \, \cdot \,,x_0), f_{+,2}(\lambda, \, \cdot \,,x_0)\big) = 1 .
\end{equation}
(Here we explicitly introduced the $x_0$ dependence of $f_{+,j}$, implied by the choice of normalization in \eqref{3.55}, as keeping track of it later on will become a necessity.)
In particular, $f_{+,1}(\lambda, \, \cdot \,,x_0)$ now plays the analog of the Jost solution in the case of a short-range potential $q$ (i.e., $q \in L^1(\bbR_+; (1 + x)dx)$, $q$ real-valued a.e.~on $\bbR_+$).

By the limit point property of $\tau_+$ at
$+ \infty$ and the asymptotic behavior of $f_{+,1}$ in \eqref{3.55} one infers, in addition,
\begin{align}
\psi_{+,\alpha}(\lambda, \, \cdot \,) &= f_{+,1} (\lambda, \, \cdot \,,x_0) \big/ \big[\sin (\alpha ) f_{+,1} ' (\lambda,0,x_0) + \cos (\alpha ) f_{+,1} (\lambda, 0,x_0)\big],   \lb{3.59} \\
\phi_{\alpha } (\lambda, \, \cdot \,) &= \big[ \cos (\alpha ) f_{+,1} (\lambda, 0,x_0) + \sin (\alpha ) f_{+,1} ' (\lambda,0,x_0) \big] f_{+,2} (\lambda, \, \cdot \,,x_0) \no\\
& \quad  - \big[\cos (\alpha )  f_{+,2} (\lambda, 0,x_0) + \sin (\alpha ) f_{+,2} ' (\lambda ,0,x_0) \big]  f_{+,1} (\lambda, \, \cdot \,,x_0)    \lb{3.59a}
\end{align}
for $\lambda < 0$ sufficiently negative. Analytic continuation with respect to $\lambda$ in
\eqref{3.59} then yields the existence of a unique Jost-type solution $f_{+,1}(z, \, \cdot \,,x_0)$ satisfying
\begin{align}
& \tau_+ f_{+,1} (z, \, \cdot \,,x_0) = z f_{+,1} (z, \, \cdot \,,x_0),
\quad z \in \bbC \backslash \bbR,   \\
& f_{+,1} (z, \, \cdot \,,x_0) \in L^2(\bbR_+; dx),  \quad z \in \bbC \backslash \bbR,
\end{align}
coinciding with $f_{+,1}(\lambda, \, \cdot \,,x_0)$ for $z = \lambda < 0$ sufficiently negative. In addition
one has
\begin{align}
W\big(f_{+,1} (z, \, \cdot \,,x_0), \phi_\alpha (z,\, \cdot \,,x_0))
&= \cos (\alpha ) f_{+,1} (z,0,x_0)+ \sin (\alpha ) f' _{+,1} (z,0,x_0), \no\\
& \hspace{4cm} z \in \rho(H_{+,\alpha} ),
\end{align}
which should be compared with the Jost function $f_+(z, 0)$ in the case where $q$ represents a short-range potential and $\alpha = 0$.

In the following we want to illustrate how Hypothesis \ref{h2.3} and \eqref{2.7} apply to
$H_{+,\alpha}$ in the case $p=2$. For this purpose we first recall the following standard convergence property for trace ideals in $\cH$:

\begin{lemma} \lb{l3.2}
Let $q \in [1,\infty)$ and assume that $R,R_n,T,T_n\in\cB(\cH)$,
$n\in\bbN$, satisfy
$\slim_{n\to\infty}R_n = R$  and $\slim_{n\to \infty}T_n = T$ and that
$S,S_n\in\cB_q(\cH)$, $n\in\bbN$, satisfy
$\lim_{n\to\infty}\|S_n-S\|_{\cB_q(\cH)}=0$.
Then $\lim_{n\to\infty}\|R_n S_n T_n^\ast - R S T^\ast\|_{\cB_q(\cH)}=0$.
\end{lemma}

\noindent 
\smallskip 
(Here the strong limit of a sequence of bounded operators $B_n$, $n \in \bbN$, as $n\to \infty$, was 
abbreviated by $\slim_{n \to \infty} B_n$.)  Lemma \ref{l3.2} follows, for instance, from 
\cite[Theorem 1]{Gr73}, \cite[p.\ 28--29]{Si05}, or
\cite[Lemma 6.1.3]{Ya92} with a minor additional effort (taking adjoints, etc.).

Next, we recall a few facts that enable one to compute the trace of a nonnegative trace class integral 
operator in a straightforward manner: \\[1mm]
$(i)$ Let $0 \leq A \in \cB(\cH)$, $\{\phi_m\}_{m \in \bbN}$ an orthonormal basis in $\cH$ (without loss of generality we assume $\dim(\cH) = \infty$), then
\begin{equation}
\sum_{m \in \bbN} (\phi_m, A \phi_m)_{\cH} \in [0,\infty) \cup \{\infty\}
\end{equation}
is independent of the orthonormal basis $\{\phi_m\}_{m \in \bbN}$ in $\cH$. Moreover, 
\begin{equation}
\sum_{m \in \bbN} (\phi_m, A \phi_m)_{\cH} < \infty \, \text{ if and only if } \, A \in \cB_1(\cH),
\end{equation}
and if $A \in \cB_1(\cH)$, 
\begin{equation}
\tr_{\cH}(A) = \sum_{m \in \bbN} (\phi_m, A \phi_m)_{\cH} = \sum_{j \in J} \lambda_j(A) = \|A\|_{\cB_1(\cH)}, 
\end{equation}
where $0 \leq \lambda_j(A)$, $j \in J$, $J \subseteq \bbN$, denote the eigenvalues of $A$ counting multiplicity.
(For details, see, e.g., \cite[Theorems~2.14 and 3.1]{Si05}.) \\[1mm]
$(ii)$ Let $0 \leq K \in \cB\big(L^2(\Omega; d^nx)\big)$, $\Omega \subseteq \bbR^n$, and suppose that $K$ is an integral operator with continuous integral kernel $K(\, \cdot \,,\, \cdot \,)$ on $\Omega \times \Omega$. Then,
\begin{align}
& 0 \leq K(x,x), \quad x \in \Omega,    \\
& |K(x,x)| \leq |K(x,x)|^{1/2} |K(x',x')|^{1/2}, \quad x, x' \in \Omega, 
\end{align} 
and for all orthonormal bases $\{e_m\}_{m \in \bbN}$ in $L^2(\Omega; d^nx)$,
\begin{equation}
\sum_{m \in \bbN} (e_m, K e_m)_{L^2(\Omega; d^nx)} = \int_{\Omega} d^n x \, K(x,x).    \lb{3.33}
\end{equation}
In particular, the finiteness of either side in \eqref{3.33} implies that of the other. Hence, 
$K \in \cB_1\big(L^2(\Omega; d^nx)\big)$ if and only if $K(\, \cdot \,,\, \cdot \,) \in L^1(\Omega; d^nx)$, and, in this case,
\begin{equation}
\tr_{L^2(\Omega; d^nx)} (K) = \int_{\Omega} d^n x \, K(x,x) < \infty.    \lb{3.34}
\end{equation}
(For more details we refer, e.g., to \cite[Proposition~5.6.9]{Da07}. For more general measure spaces see, e.g., \cite[p.~65--66]{RS79}, \cite[Sect.~3.11]{Si15}.) 

Next, we introduce the family of self-adjoint projections $P_R$ in $L^2(\bbR_+; dx)$ via
\begin{equation}
(P_R f)(x) = \chi_{[0,R]}(x) f(x), \quad f \in L^2(\bbR_+; dx), \, R > 0,
\end{equation}
with $\chi_{[0,R]}(\, \cdot \,)$ the characteristic function associated with the interval
$[0,R]$, $R > 0$. ($P_R$ will play the role of $R_n, T_n$ in our application of Lemma \ref{l3.2}
in the proof of Theorem \ref{t3.3} below.)

One then obtains the following results.

\begin{theorem} \lb{t3.3}
Assume Hypothesis \ref{h3.1}, $z, z_0 \in \rho(H_{+,\alpha})$, and $\alpha \in [0,\pi)$. Then,
\begin{equation}
\big[(H_{+,\alpha} -z I_+)^{-1} - (H_{+,\alpha} -z_0 I_+)^{-1}\big] \in \cB_1\big(L^2(\bbR_+; dx)\big),     \lb{3.60}
\end{equation}
and
\begin{align}
& \tr_{L^2(\bbR_+;dx)} \big((H_{+,\alpha} - z I_+)^{-1} - (H_{+,\alpha} - z_0 I_+)^{-1}\big)    \no \\
& \quad = - \f{d}{dz} \ln\big(\det{_{2,L^2(\bbR_+;dx)}} \big(I_+ - (z - z_0) (H_{+,\alpha } - z_0 I_+)^{-1}\big)\big)
\no \\
& \quad = \f{d}{dz} \ln\big(\sin (\alpha) f_{+,1} ' (z,0,x_0)+\cos (\alpha )  f_{+,1}(z,0,x_0)\big)\bigg|_{z=z_0} \no\\
& \qquad - \f{d}{dz} \ln\big(\sin (\alpha ) f_{+,1}' (z,0,x_0) + \cos (\alpha ) f_{+,1}(z,0,x_0)\big)    \no \\
& \qquad +\f{1}{2}  \cI (z,z_0,x_0),    \lb{3.65}
\end{align}
as well as,
\begin{align}
& \det{_{2,L^2(\bbR_+;dx)}} \big(I_+ - (z - z_0) (H_{+,\alpha } - z_0 I_+)^{-1}\big)    \no \\
& \quad = \bigg[\f{ \sin (\alpha ) f_{+,1} ' (z,0,x_0) + \cos (\alpha ) f_{+,1} (z,0,x_0) }{ \sin (\alpha ) f_{+,1} ' (z_0,0,x_0) + \cos (\alpha ) f_{+,1} (z_0,0,x_0)}\bigg]   \no \\
& \qquad \times \exp{\bigg( - (z - z_0) \,\,\f{ \sin (\alpha ) \dot{f^{\, \prime}}_{\!\!\!\!+,1} (z_0,0,x_0)+\cos (\alpha ) \dot{f}_{+,1} (z_0,0,x_0) }
{ \sin (\alpha ) f_{+,1}'(z_0,0,x_0)+\cos (\alpha ) f_{+,1} (z_0, 0,x_0) } \bigg)}    \lb{3.66}  \\
& \qquad \times \exp\bigg(- \f{1}{2} \int_{z_0}^z d \zeta \, \cI(\zeta,z_0,x_0)\bigg),    \no
\end{align}
where we abbreviated $\dot{} = d/dz$ and 
\begin{equation}
\cI(z,z_0,x_0) = \int_{x_0}^{\infty} dx \big\{[q(x) - z]^{-1/2} - [q(x) - z_0]^{-1/2}\big\}.    \lb{3.66a}
\end{equation}
\end{theorem}
\begin{proof}
Since the resolvents of $H_{+,\alpha}$, $\alpha \in (0,\pi)$, and $H_{+,0}$ differ only by a rank-one operator, it suffices 
to choose $\alpha = 0$ when establishing \eqref{3.60}.

We will first prove \eqref{3.60} for $z=\lambda < 0$, $z_0 = \lambda_0 < \lambda < 0$, and  employ 
monotonicity of resolvents with respect to $\lambda < 0$ sufficiently negative, implying 
\begin{equation}
0 \leq \big[(H_{+,0} -\lambda I_+)^{-1} - (H_{+,0} - \lambda_0 I_+)^{-1}\big], \quad
\lambda_0 < \lambda < 0,     \lb{3.66b} 
\end{equation}
with $\lambda < 0$ sufficiently negative (the latter will be assumed for most of the remainder of this proof). Subsequently, we will apply \eqref{3.34} to $K$ given by the right-hand side of inequality \eqref{3.66b}.

Equations \eqref{3.59} and \eqref{3.59a} yield for $\alpha =0$,
\begin{align}
\begin{split}
\phi_0 (\lambda, \, \cdot \,) \psi_{+,0} (\lambda, \, \cdot \,)
&= f_{+,1} (\lambda, \, \cdot \,,x_0) f_{+,2} (\lambda, \, \cdot \,,x_0)    \\
& \quad - f_{+,1} (\lambda, 0,x_0)^{-1} f_{+,2} (\lambda, 0,x_0)
f_{+,1} (\lambda, \, \cdot \,,x_0)^2,    \lb{3.71}
\end{split}
\end{align}
and since by \eqref{3.55} for $j=1$ integrability properties of \eqref{3.71} over $\bbR_+$
depend on those of $f_{+,1} (\lambda, \, \cdot \, ,x_0 \,) f_{+,2} (\lambda, \, \cdot \, ,x_0 \,)$, we now investigate
the latter on $[x_0, \infty)$. Employing \eqref{3.55} once more then yields
\begin{align}
& 0 \leq [\phi_0 (\lambda, x) \psi_{+,0} (\lambda, x)
- \phi_0 (\lambda_0, x) \psi_{+,0} (\lambda_0, x)]     \no \\
& \underset{x \to \infty}{=} 2^{-1} \big\{[q(x) - \lambda]^{-1/2} - [q(x) - \lambda_0]^{-1/2}\big\} [1 + \oh(1)]
\no \\
& \underset{x \to \infty}{=} 4^{-1} (\lambda - \lambda_0) q(x)^{-3/2} [1 + \oh(1)]    \no \\
& \underset{x \to \infty}{=} 4^{-1} (\lambda - \lambda_0) \, C_0 \, x^{-1 - (3 \varepsilon_0/2)} [1 + \oh(1)],
\end{align}
according to \eqref{3.52}, proving integrability of $[\phi_0 (\lambda, \, \cdot \,) \psi_{+,0} (\lambda,  \, \cdot \,)
- \phi_0 (\lambda_0,  \, \cdot \,) \psi_{+,0} (\lambda_0,  \, \cdot \,)] $ near $+\infty$. An application of \eqref{3.34} to $K$ given by the right-hand side of inequality \eqref{3.66b} then yields \eqref{3.60} with $z=\lambda$, $z_0 = \lambda_0 < \lambda < 0$. Analytic continuation in $\lambda$ and subsequently in $\lambda_0$ proves \eqref{3.60}. 

By \eqref{2.7} with $p=2$ this proves the first equality in \eqref{3.65}.

To prove the second equality in \eqref{3.65}, we now apply Lemma \ref{l3.2} in the trace class
case $q=1$ and combine it with \eqref{3.60} to arrive at
\begin{align}
& \tr_{L^2(\bbR_+;dx)} \big((H_{+,\alpha } - \lambda I_+)^{-1} - (H_{+,\alpha } - \lambda_0 I_+)^{-1}\big)    \no \\
& \quad = \lim_{R \to \infty}
\tr_{L^2(\bbR_+;dx)} \big(P_R \big[(H_{+,\alpha } - \lambda I_+)^{-1} - (H_{+,\alpha } - \lambda_0 I_+)^{-1}\big] P_R\big)
\no \\
& \quad = \lim_{R \to \infty} \int_0^R dx \, [\phi_\alpha  (\lambda, x) \psi_{+,\alpha } (\lambda, x)
- \phi_\alpha (\lambda_0, x) \psi_{+,\alpha } (\lambda_0, x)]    \no \\
& \quad = \lim_{R \to \infty}
\big[W\big(\phi_\alpha (\lambda_0, \, \cdot \,), \dot \psi_{+,\alpha } (\lambda_0, \, \cdot \,) \big)(R)
- W\big(\phi_\alpha (\lambda, \, \cdot \,), \dot \psi_{+,\alpha } (\lambda, \, \cdot \,) \big)(R) \big]    \no \\
& \qquad + W\big(\phi_\alpha (\lambda, \, \cdot \,), \dot \psi_{+,\alpha } (\lambda, \, \cdot \,) \big)(0)
- W\big(\phi_\alpha (\lambda_0, \, \cdot \,), \dot \psi_{+,\alpha } (\lambda_0, \, \cdot \,) \big)(0)   \no \\
& \quad = \lim_{R \to \infty}
\big[W\big(\phi_\alpha (\lambda_0, \, \cdot \,), \dot \psi_{+,\alpha } (\lambda_0, \, \cdot \,) \big)(R)
- W\big(\phi_\alpha (\lambda, \, \cdot \,), \dot \psi_{+,\alpha } (\lambda, \, \cdot \,) \big)(R) \big],  \lb{3.39}
\end{align}
since
\begin{align}
& W \big(\phi _\alpha (\lambda , \cdot),\dot{\psi}_{+,\alpha} (\lambda , \cdot ) \big) (0) = - \sin (\alpha ) \dot{\psi'} _{\!\!\!+,\alpha }  (\lambda ,0) -\cos (\alpha ) \dot{\psi} _{+,\alpha } (\lambda , 0) \no\\
&\qquad = - \f{d}{d\lambda} \big[ \sin (\alpha ) \psi_{+,\alpha } ' (\lambda , 0) + \cos (\alpha ) \psi_{+,\alpha } (\lambda ,0) \big] =0.
\end{align}
It remains to analyze the right-hand side of \eqref{3.39}. To this end we note that
\begin{equation}
\tau_+ \dot{f}_{+,1} (z,x,x_0) = z \dot{f}_{+,1} (z,x,x_0) + f_{+,1} (z,x,x_0),
\end{equation}
and hence
\begin{align}
\dot{f}_{+,1} (z,x,x_0) &= c_1(z) f_{+,1} (z,x,x_0) + c_2(z) f_{+,2} (z,x,x_0)   \no \\
& \quad + f_{+,1} (z,x,x_0) \int_0^x dx' \, f_{+,1} (z,x',x_0) f_{+,2} (z,x',x_0)   \lb{3.78} \\
& \quad - f_{+,2} (z,x,x_0) \int_0^x dx' \, f_{+,1} (z,x',x_0)^2,     \no \\
\dot{f^{\; \prime}}_{\!\!\!\!+,1} (z,x,x_0) &= c_1(z) f^{\; \prime}_{+,1} (z,x,x_0)
+ c_2(z) f^{\; \prime}_{+,2} (z,x,x_0)    \no \\
& \quad + f^{\; \prime}_{+,1} (z,x,x_0) \int_0^x dx' \, f_{+,1} (z,x',x_0) f_{+,2} (z,x',x_0)
 \lb{3.79} \\
& \quad - f^{\; \prime}_{+,2} (z,x,x_0) \int_0^x dx' \, f_{+,1} (z,x',x_0)^2.   \no
\end{align}
Next, we claim that
\begin{equation}
c_2(z) = \int_0^{\infty} dx' \, f_{+,1} (z,x',x_0)^2, \quad z \in \rho(H_{+,\alpha }),    \lb{3.80}
\end{equation}
and hence \eqref{3.78}, \eqref{3.79} simplify to
\begin{align}
\dot{f}_{+,1} (z,x,x_0) &= c_1(z) f_{+,1} (z,x,x_0)     \no \\
& \quad + f_{+,1} (z,x,x_0) \int_0^x dx' \, f_{+,1} (z,x',x_0) f_{+,2} (z,x',x_0)
\lb{3.81} \\
& \quad + f_{+,2} (z,x,x_0) \int_x^{\infty} dx' \, f_{+,1} (z,x',x_0)^2,     \no \\
\dot{f^{\; \prime}}_{\!\!\!\!+,1} (z,x,x_0) &= c_1(z) f^{\; \prime}_{+,1} (z,x,x_0)    \no \\
& \quad + f^{\; \prime}_{+,1} (z,x,x_0) \int_0^x dx' \, f_{+,1} (z,x',x_0) f_{+,2} (z,x',x_0)
\lb{3.82} \\
& \quad + f^{\; \prime}_{+,2} (z,x,x_0) \int_x^{\infty} dx' \, f_{+,1} (z,x',x_0)^2.     \no
\end{align}
To infer the necessity of \eqref{3.80} one can argue by contradiction as follows: If
\eqref{3.80} does not hold, then integrating $\dot{f}_{+,1} (z,x)$ with respect
to $z$ from $\lambda_0$ to $\lambda$ along the negative real axis on the left-hand side
of \eqref{3.78} yields
\begin{equation}
\int_{\lambda_0}^{\lambda} dz \, \dot{f}_{+,1} (z,x,x_0) = f_{+,1} (\lambda,x,x_0)
- f_{+,1} (\lambda_0,x,x_0) \underset{x \to \infty}{\longrightarrow} 0    \lb{3.83}
\end{equation}
by the first asymptotic relation in \eqref{3.55}. However,
with \eqref{3.80} violated, integrating the right-hand side of \eqref{3.78} with respect
to $z$ from $\lambda_0$ to $\lambda$ along the negative real axis now yields several contributions
vanishing as $x \to \infty$ (again invoking \eqref{3.55}), but there will also be one integral of the
type
\begin{equation}
\int_{\lambda_0}^{\lambda} dz \, f_{+,2} (z,x,x_0) A(z,x)
\underset{x \to \infty}{\centernot \longrightarrow} 0     \lb{3.84}
\end{equation}
where $A(z,\, \cdot \,)$ is bounded with a finite nonzero limit, $\lim_{x \to \infty} A(z,x) = A(z,\infty) \neq 0$.
Relation \eqref{3.84} contradicts \eqref{3.83}, proving \eqref{3.80}.

Investigating the asymptotics of the right-hand sides of \eqref{3.81}, \eqref{3.82}, invoking the
leading asymptotic behavior \eqref{3.55}, then shows that to obtain the leading asymptotic behavior
of $\dot{f}_{+,1} (\lambda,x,x_0)$, $\dot{f}_{+,2} (\lambda,x,x_0)$ one can formally differentiate relations
\eqref{3.55} with respect to $\lambda$ and hence obtains,
\begin{align}
\begin{split}
& \dot{f}_{+,1} (\lambda,x,x_0) \underset{x \to \infty}{=} 2^{-3/2} [q(x) - \lambda]^{-1/4}
\int_{x_0}^x dx''\, [q(x'') - \lambda]^{-1/2}   \\
& \hspace*{2.9cm} \times \exp\bigg(- \int_{x_0}^x dx' [q(x') - \lambda]^{1/2}\bigg) [1 + \oh(1)],  \\
& \dot{f^{\; \prime}}_{\!\!\!\!+,1} (\lambda,x,x_0) \underset{x \to \infty}{=} - 2^{-3/2} [q(x) - \lambda]^{1/4}
\int_{x_0}^x dx''\, [q(x'') - \lambda]^{-1/2}   \lb{3.85} \\
& \hspace*{3cm} \times \exp\bigg(- \int_{x_0}^x dx' [q(x') - \lambda]^{1/2}\bigg) [1 + \oh(1)],
\end{split}
\end{align}
for $\lambda < 0$ sufficiently negative according to our convention in this proof.

Next, one utilizes \eqref{3.59} and \eqref{3.59a} and computes
\begin{align}
& W\big(\phi_\alpha (\lambda, \, \cdot \,), \dot \psi_{+,\alpha } (\lambda, \, \cdot \,) \big)(R)   \no \\
& \quad \underset{R \to \infty}{=} f_{+,2} (\lambda,R,x_0)
\dot{f^{\; \prime}}_{\!\!\!\!+,1} (\lambda,R,x_0)   \no \\
& \qquad \quad \; - f_{+,2} (\lambda,R,x_0) f_{+,1} ' (\lambda,R,x_0) \, \f{\sin (\alpha ) \dot{f^{\; \prime}}_{\!\!\!\!+,1}(\lambda,0,x_0)+\cos (\alpha ) \dot{f}_{+,1} (\lambda , 0,x_0) }
{\sin (\alpha ) f_{+,1} ' (\lambda , 0,x_0) + \cos (\alpha ) f_{+,1} (\lambda , 0,x_0) }   \no \\
& \qquad \quad \; -  f^{\; \prime}_{+,2}(\lambda,R,x_0) \dot{f}_{+,1} (\lambda,R,x_0)  \no \\
& \qquad \quad \;
+ f^{\; \prime}_{+,2}(\lambda,R,x_0) f_{+,1} (\lambda , R,x_0)  \, \f{\sin (\alpha ) \dot{f^\prime}_{\!\!\!\!+,1}(\lambda,0,x_0)+\cos (\alpha ) \dot{f}_{+,1} (\lambda , 0,x_0)}
{\sin (\alpha ) f'_{+,1} (\lambda,0,x_0) +\cos (\alpha ) f_{+,1} (\lambda , 0,x_0)}  \no \\
& \quad \underset{R \to \infty}{=} \dot{f^{\; \prime}}_{\!\!\!\!+,1} (\lambda,R,x_0) f_{+,2} (\lambda,R,x_0)
- \dot{f}_{+,1} (\lambda,R,x_0) f^{\; \prime}_{+,2}(\lambda,R,x_0)    \no \\
& \qquad \quad \; + \f{\sin (\alpha ) \dot{f'}_{\!\!\!\!+,1} (\lambda,0,x_0)+\cos (\alpha )\dot{f}_{+,1} (\lambda , 0,x_0)}
{\sin (\alpha ) f_{+,1} ' (\lambda , 0,x_0) + \cos (\alpha ) f_{+,1} (\lambda,0,x_0)},    \lb{3.86}
\end{align}
again for $\lambda < 0$ sufficiently negative. Insertion of \eqref{3.55} and \eqref{3.85} into
\eqref{3.86} finally implies
\begin{align}
W\big(\phi_\alpha (\lambda, \, \cdot \,), \dot \psi_{+,\alpha } (\lambda, \, \cdot \,) \big)(R)
& \underset{R \to \infty}{=}  \f{\sin (\alpha ) \dot{f'}_{\!\!\!\!+,1} (\lambda,0,x_0)+\cos (\alpha )\dot{f}_{+,1} (\lambda , 0,x_0)}
{\sin (\alpha ) f_{+,1} ' (\lambda , 0,x_0) + \cos (\alpha ) f_{+,1} (\lambda,0,x_0)}   \no \\[1mm]
& \qquad \;\; - 2^{-1} \bigg(\int_{x_0}^R dx \, [q(x) - \lambda]^{- 1/2}\bigg) [1 + \oh(1)].    \label{3.86a}
\end{align}
Returning to \eqref{3.39} this yields
\begin{align}
& \tr_{L^2(\bbR_+;dx)} \big((H_{+,\alpha } - \lambda I_+)^{-1} - (H_{+,\alpha } - \lambda_0 I_+)^{-1}\big)    \no \\
& \quad = \lim_{R \to \infty}
\big[W\big(\phi_\alpha (\lambda_0, \, \cdot \,), \dot \psi_{+,\alpha } (\lambda_0, \, \cdot \,) \big)(R)
- W\big(\phi_\alpha (\lambda, \, \cdot \,), \dot \psi_{+,\alpha } (\lambda, \, \cdot \,) \big)(R) \big],  \no \\
& \underset{R \to \infty}{=}  \f{\sin (\alpha ) \dot{f'}_{\!\!\!+,1} (\lambda_0,0,x_0)+\cos (\alpha )\dot{f}_{+,1} (\lambda_0 , 0,x_0)}
{\sin (\alpha ) f_{+,1} ' (\lambda_0 , 0,x_0) + \cos (\alpha ) f_{+,1} (\lambda_0,0,x_0)}  \no \\[1mm]
&\qquad -\f{\sin (\alpha ) \dot{f'}_{\!\!\!+,1} (\lambda,0,x_0)+\cos (\alpha )\dot{f}_{+,1} (\lambda , 0,x_0)}
{\sin (\alpha ) f_{+,1} ' (\lambda , 0,x_0) + \cos (\alpha ) f_{+,1} (\lambda,0,x_0)} \no \\[1mm]
& \qquad \;\; + 2^{-1} \bigg(\int_{x_0}^R dx \,
\big\{ [q(x) - \lambda]^{-1/2} - [q(x) - \lambda_0]^{-1/2}\big\}\bigg) [1 + \oh(1)]    \no \\
& \;\;\, = \;\; \f{\sin (\alpha ) \dot{f'}_{\!\!\!+,1} (\lambda_0,0,x_0)+\cos (\alpha )\dot{f}_{+,1} (\lambda_0 , 0,x_0)}
{\sin (\alpha ) f_{+,1} ' (\lambda_0 , 0,x_0) + \cos (\alpha ) f_{+,1} (\lambda_0,0,x_0)}  \no \\[1mm]
&\qquad -\f{\sin (\alpha ) \dot{f'}_{\!\!\!+,1} (\lambda,0,x_0)+\cos (\alpha )\dot{f}_{+,1} (\lambda , 0,x_0)}
{\sin (\alpha ) f_{+,1} ' (\lambda , 0,x_0) + \cos (\alpha ) f_{+,1} (\lambda,0,x_0)} \no \\[1mm]
& \qquad \;\; + 2^{-1} \bigg(\int_{x_0}^{\infty}  dx \,
\big\{ [q(x) - \lambda]^{-1/2} - [q(x) - \lambda_0]^{-1/2}\big\}\bigg),     \lb{3.89}
\end{align}
and hence \eqref{3.65} for $z = \lambda < 0$, $z_0 = \lambda_0 < 0$, both sufficiently negative.
In this context one observes that for $x_0 > 0$ sufficiently large,
\begin{align}
\begin{split}
& 2^{-1} \bigg(\int_{x_0}^R dx \,
\big\{ [q(x) - \lambda]^{-1/2} - [q(x) - \lambda_0]^{-1/2}\big\}\bigg)     \\
& \quad \underset{R \to \infty}{=}  \f{1}{4} (\lambda - \lambda_0) \bigg(\int_{x_0}^R dx \,
q(x) ^{-3/2}\bigg) [1 +\oh(1)]
\end{split}
\end{align}
with $q^{-3/2} \in L^1([x_0,\infty); dx)$ by Hypothesis \eqref{3.52}.

Analytic continuation in $\lambda$ of both sides in \eqref{3.89} extends the latter to $z \in \rho(H_{+,\alpha })$.
Similarly, analytic continuation in $\lambda_0$ of both sides in \eqref{3.89} extends the latter to
$z_0 \in \rho(H_{+,\alpha })$, completing the proof of \eqref{3.65}.

Relation \eqref{3.66} then follows from integrating \eqref{3.65} with respect to the energy variable  from $z_0$ to $z$.
\end{proof}

\begin{remark} \lb{r3.4a} 
Employing the resolvent equation,
\begin{align}
& (H_{+,0} - z I_+)^{-1} - (H_{+,0} - z_0 I_+)^{-1} = (z - z_0)
(H_{+,0} - z I_+)^{-1} (H_{+,0} - z_0 I_+)^{-1},     \no \\
& \hspace*{8.5cm}  z, z_0 \in \rho(H_{+,0}),
\end{align}
an alternative proof of relation \eqref{3.60} follows upon establishing
\begin{equation}
(H_{+,0} - z I_+)^{-1} \in \cB_2\big(L^2(\bbR_+; dx)\big), \quad  z \in \rho(H_{+,0}).  \lb{3.67}
\end{equation}
To prove \eqref{3.67} in turn it suffices to establish the Hilbert--Schmidt property for some
$z=\lambda < 0$ sufficiently negative, followed by analytic continuation with respect to $z \in \rho(H_{+,0})$. 
Given the Green's function of $H_{+,0}$ in \eqref{3.57},
it thus suffices to prove that
\begin{equation}
\int_{\bbR_+} \int_{\bbR_+}  dx \, dx' \, |\phi_{0}(\lambda,x) \, \psi_{+,0} (\lambda,x')|^2 < \infty; 
\end{equation}
we omit further details at this point.
\end{remark}

Next, we apply Theorem \ref{t3.3} to the following explicitly solvable example concerning the linear potential
and denote by $Ai(\, \cdot \,), Bi(\, \cdot \,)$ the Airy functions as discussed, for instance, in
\cite[Sect.~10.4]{AS72}.

\begin{example} \lb{e3.4}
Consider the special case $q(x) = x$, $x \in \bbR_+$, and $\alpha = 0$. Then,
for $x \in \bbR_+$, $z, z_0 \in \rho(H_{+,0})$,
\begin{align}
& f_{+,1} (z,x,x_0) = (2 \pi)^{1/2} e^{(2/3) (x_0 - z)^{3/2}} Ai(x-z),    \\
& f_{+,2} (z,x,x_0) = (\pi /2)^{1/2} e^{- (2/3) (x_0 - z)^{3/2}} Bi(x-z),    \\
& W(f_{+,1} (z,\, \cdot \,,x_0), f_{+,2} (z,\, \cdot \,,x_0)) =1,   \\
& \phi_0(z,x) = \pi [Ai(-z) Bi(x-z) - Bi(-z) Ai(x-z)],    \\
& \psi_{+,0} (z,x) = Ai(x-z)/Ai(-z),    \\
& W(\phi_0(z,\, \cdot \,), \dot{\psi}_{+,0} (z,\, \cdot \,))(x)     \\
& \quad = \pi [Ai'(x-z) Bi'(x-z) - (x-z) Ai(x-z) Bi(x-z)] - [Ai'(-z)/Ai(-z)],    \no \\
\begin{split}
& \cI(z,z_0,x_0) = \int_{x_0}^{\infty} dx \big\{[x - z]^{-1/2} - [x - z_0]^{-1/2}\big\}   \\[1mm]
& \hspace*{1.65cm} = 2 \big[(x_0 - z_0)^{1/2} - (x_0 - z)^{1/2}\big],
\end{split} \\
\begin{split}
& \tr_{L^2(\bbR_+;dx)} \big((H_{+,0} - z I_+)^{-1} - (H_{+,0} - z_0 I_+)^{-1}\big)     \\
& \quad = \psi_{+,0}'(z,0) - \psi_{+,0}'(z_0,0) = [Ai'(-z)/Ai(-z)] - [Ai'(-z_0)/Ai(-z_0)],
\end{split} \\
\begin{split}
& \det{_{2,L^2(\bbR_+;dx)}} \big(I_+ - (z - z_0) (H_{+,0} - z_0 I_+)^{-1}\big)    \\
& \quad = [Ai(-z)/Ai(-z_0)] \exp\big((z-z_0) [Ai'(-z_0)/Ai(-z_0)]\big).    \lb{3.99}
\end{split}
\end{align}
\end{example}

We note that \eqref{3.99} was recently considered in \cite{Me16}, but the exponential factor
in \eqref{3.99} was missed in \cite{Me16}.

Finally, we generalize Theorem \ref{t3.3} to the following setting.

\begin{theorem} \lb{t3.5}
Assume Hypothesis \ref{h3.1}, $z\in \rho(H_{+,\alpha_2})$, $z_0 \in \rho (H_{+, \alpha_1})$, and $\alpha_1,\alpha_2 \in [0,\pi)$. Then,
\begin{equation}
\big[(H_{+,\alpha_2} -z I_+)^{-1} - (H_{+,\alpha_1} -z_0 I_+)^{-1}\big] \in \cB_1\big(L^2(\bbR_+; dx)\big),     \lb{3.74a}
\end{equation}
and $($cf.\ \eqref{3.66a}$)$
\begin{align}
& {\tr}_{L^2(\bbR_+; dx)} \big((H_{+,\alpha_2} - z I_+)^{-1}
- (H_{+,\alpha_1} - z_0 I_+)^{-1}\big)    \no \\
& \quad = - \f{d}{dz} \ln\bigg(\f{\sin(\alpha_2) f_{+,1}'(z,0,x_0)
+ \cos(\alpha_2) f_{+,1}(z ,0,x_0)}{\sin(\alpha_1) f_{+,1}'(z_0 ,0,x_0)
+ \cos(\alpha_1) f_{+,1}(z_0 ,0,x_0)}\bigg),    \lb{3.75} \\
& \qquad +\f{1}{2} \cI(z,z_0,x_0) .     \no
\end{align}
\end{theorem}
\begin{proof}
Equation \eqref{3.74a} is established exactly as in the proof of Theorem \ref{t3.3}. Furthermore, as argued there one has
\begin{align}
& \tr_{L^2(\bbR_+;dx)} \big((H_{+,\alpha_2 } - \lambda I_+)^{-1} - (H_{+,\alpha_1 } - \lambda_0 I_+)^{-1}\big)    \no \\
& \quad = \lim_{R \to \infty}
\big[W\big(\phi_{\alpha_1} (\lambda_0, \, \cdot \,), \dot \psi_{+,\alpha_1 } (\lambda_0, \, \cdot \,) \big)(R)
- W\big(\phi_{\alpha_2} (\lambda, \, \cdot \,), \dot \psi_{+,\alpha_2 } (\lambda, \, \cdot \,) \big)(R) \big] .
\end{align}
Using equation \eqref{3.86a} then immediately implies \eqref{3.75}.
\end{proof}

Setting $z=z_0$, we obtain in particular
\begin{align}
& {\tr}_{L^2(\bbR_+; dx)} \big((H_{+,\alpha_2} - z I_+)^{-1}
- (H_{+,\alpha_1} - z I_+)^{-1}\big)    \no \\
& \quad = - \f{d}{dz} \ln\bigg(\f{\sin(\alpha_2) f_{+,1}'(z,0,x_0)
+ \cos(\alpha_2) f_{+,1}(z ,0,x_0)}{\sin(\alpha_1) f_{+,1}'(z ,0,x_0)
+ \cos(\alpha_1) f_{+,1}(z ,0,x_0)}\bigg) .    \lb{3.76}
\end{align}

\begin{remark}
In order to proof Theorem \ref{t3.5}, one could instead have proven the slightly simpler result \eqref{3.76}
and then note that
\begin{align}
& {\tr}_{L^2(\bbR_+; dx)} \big((H_{+,\alpha_2} - z I_+)^{-1}
- (H_{+,\alpha_1} - z_0 I_+)^{-1}\big)    \no \\
& = {\tr}_{L^2(\bbR_+; dx)} \big((H_{+,\alpha_2} - z I_+)^{-1}
- (H_{+,\alpha_1} - z I_+)^{-1}\big)    \no \\
& \quad + {\tr}_{L^2(\bbR_+; dx)} \big((H_{+,\alpha_1} - z I_+)^{-1}
- (H_{+,\alpha_1} - z_0 I_+)^{-1}\big),
\end{align}
which, using \eqref{3.76} together with Theorem \ref{t3.3}, implies Theorem \ref{t3.5}.
\hfill $\diamond$
\end{remark}

\medskip
\noindent {\bf Acknowledgments.} We are indebted to the referee for a very careful reading of our manuscript. 


\end{document}